\documentclass[10pt,reqno]{amsart}

\usepackage{amsmath,amssymb,amsfonts,amsthm}
\usepackage{enumerate}
\usepackage{url}
\usepackage[breaklinks=true,colorlinks=true,linkcolor=blue,citecolor=blue!60!green,filecolor=blue,urlcolor=blue]{hyperref}
\hypersetup{linktocpage}
\usepackage{bm}
\usepackage[hiresbb]{graphicx,xcolor}
\usepackage{tcolorbox}
\usepackage{booktabs,tabularx}
\usepackage{mathtools}
\mathtoolsset{showonlyrefs=true}
\usepackage{mathrsfs}
\usepackage{float}
\usepackage{subcaption}
\usepackage{yfonts}
\usepackage{accents}
\usepackage{here}
\usepackage{hhline}
\allowdisplaybreaks[3]
\usepackage[top=30truemm,bottom=30truemm,left=30truemm,right=30truemm]{geometry}

\makeatletter
\newcommand{\tpmod}[1]{{\@displayfalse\pmod{#1}}}
\makeatother
\def\rnum#1{\expandafter{\romannumeral #1}} 
\def\Rnum#1{\uppercase\expandafter{\romannumeral #1}}
\makeatletter
\let\amsmath@bigm\bigm
\renewcommand{\bigm}[1]{%
\ifcsname fenced@\string#1\endcsname
\expandafter\@firstoftwo
\else
\expandafter\@secondoftwo
\fi
{\expandafter\amsmath@bigm\csname fenced@\string#1\endcsname}%
{\amsmath@bigm#1}%
}
\newcommand{\DeclareFence}[2]{\@namedef{fenced@\string#1}{#2}}
\makeatother
\DeclareFence{\mid}{|}

\newcommand{\C}{\mathbb{C}}

\newcommand{\Q}{\mathbb{Q}}

\newcommand{\bsm}{\left(\begin{smallmatrix}}
\newcommand{\esm}{\end{smallmatrix} \right)}
\newcommand{\bpm}{\begin{pmatrix}}
\newcommand{\epm}{\end{pmatrix}}

\newtheorem{theorem}{Theorem}[section]
\newtheorem{lemma}[theorem]{Lemma}

\newtheorem{conjecture}[theorem]{Conjecture}

\theoremstyle{definition}

\numberwithin{equation}{section}

\newcommand{\Li}{\mathrm{Li}}

\usepackage{newtxtext,newtxmath}
\usepackage{eucal}

\begin{document}

\title{Euler Product Asymptotics for Dirichlet $L$-Functions}

\author{Ikuya Kaneko}
\address{Department of Mathematics, California Institute of Technology, 1200 E California Blvd, Pasadena, CA 91125, USA}
\email{ikuyak@icloud.com}
\urladdr{\href{https://sites.google.com/view/ikuyakaneko/}{https://sites.google.com/view/ikuyakaneko/}}
\thanks{The author is supported in part by the Masason Foundation and the Spirit of Ramanujan STEM Talent Initiative.}



\date{}

\dedicatory{}

\begin{abstract}
Via the work of Ramanujan, we establish the asymptotic behaviour of partial Euler products for Dirichlet $L$-functions under the Generalised Riemann Hypothesis (GRH). Understanding the behaviour of Euler~products on the critical line is called the Deep Riemann Hypothesis (DRH). This work manifests the relation between GRH~and~DRH.
\end{abstract}

\maketitle

\section{Introduction}\label{Introduction}

\subsection{Overview and motivation}\label{overview-and-motivation}
This work is motivated by the beautiful work of Ramanujan on asymptotics~for the partial Euler product of the Riemann zeta function $\zeta(s)$. We handle the family of Dirichlet $L$-functions
\begin{equation*}
L(s, \chi) = \sum_{n = 1}^{\infty} \chi(n) n^{-s} = \prod_{p} (1-\chi(p) p^{-s})^{-1} \quad \text{with} \quad \Re(s) > 1,
\end{equation*}
and aim at proving the asymptotic behaviour of partial Euler products
\begin{equation}\label{partial}
\prod_{p \leqslant x} (1-\chi(p) p^{-s})^{-1}
\end{equation}
in the critical strip $0 < \Re(s) < 1$ with recourse to the Generalised Riemann Hypothesis (GRH) for this~family.~In 1984, Mertens~\cite{Mertens1874} conceived of the partial Euler products for $\zeta(s)$ and $L(s, \chi_{4})$ at $s = 1$, where $\chi_{4}$ is~the~primitive character modulo 4. The $\sqrt{2}$ phenomenon occurs at the central point~$s = 1/2$, which was observed by Conrad~\cite{Conrad2005}.

For technical convenience, let $\chi$ modulo $q$ be a primitive character throughout this article. Let $\varphi(q)$ be~Euler's totient function and let $\Lambda(n)$ be the von Mangoldt function. We then define the allied counting functions
\begin{equation*}
\vartheta(x; q, a) = \sum_{\substack{p \leqslant x \\ p \equiv a \tpmod q}} \log p \qquad \text{and} \qquad
\psi(x; a, q) = \sum_{\substack{n \leqslant x \\ n \equiv a \tpmod q}} \Lambda(n).
\end{equation*}
To understand the behaviour of~\eqref{partial} as $x$ tends to infinity, we shall follow an idea of Ramanujan~\cite{Ramanujan1997}~and~utilise his technique. A feature of his method is to use an accurate version of the explicit formula. This~was~created~in the process of studying the maximal order of the divisor function by introducing highly composite numbers.~The aim of this article is to generalise the formula due to Ramanujan to the context of Dirichlet $L$-functions.
\begin{theorem}\label{aim}
Let $\chi$ be a primitive Dirichlet character modulo $q$. Write $s = \sigma+it$ with $\Re(s) > 0$ and
\begin{equation*}
S_{s}(x, \chi) = -\frac{s}{\varphi(q)} \ \sideset{}{^{\ast}} \sum_{a \tpmod q} \chi(a) 
\sum_{\psi \tpmod q} \overline{\psi}(a) \sum_{L(\rho, \psi) = 0} \dfrac{x^{\rho-s}}{\rho(\rho-s)}.
\end{equation*}
If GRH for a Dirichlet~$L$-function associated to $\chi$ is assumed, we then have for $q \leqslant \sqrt{x}/(\log x)^{2+\epsilon}$ that
\begin{description}
\item[Case I. $0 < \Re(s) < 1/2$] 
\begin{multline*}
\prod_{p \leqslant x} (1-\chi(p) p^{-s})^{-1}
 = \exp \Bigg(\frac{1}{\varphi(q)} \ \sideset{}{^{\ast}} \sum_{a \tpmod q} \chi(a) \Li((\varphi(q) \vartheta(x; q, a))^{1-s})
 - \frac{1}{\varphi(q)} \Li(x^{1-\varphi(q)s})-\dotsb\\
 - \frac{1}{n} \Li(x^{1-ns})+\frac{(2s-1+\delta_{\chi^{2} = 1})x^{1/2-s}}{(1-2s) \log x}
 - \frac{S_{s}(x, \chi)}{\log x}+O \left(\frac{x^{1/2-\sigma}}{(\log x)^{2}} \right) \Bigg)
\end{multline*}
with $n$ the largest multiple of $\varphi(q)$ not exceeding $[1+1/2\sigma]$,\\
\item[Case I\hspace{-.1mm}I. $\Re(s) = 1/2$] 
\begin{multline*}
\prod_{p \leqslant x} (1-\chi(p) p^{-s})^{-1}\\
 = L(s, \chi) \exp \Bigg(\frac{1}{\varphi(q)} \ \sideset{}{^{\ast}} \sum_{a \tpmod q} \chi(a) \Li((\varphi(q) \vartheta(x; q, a))^{1-s})
 + \frac{x^{1/2-s}+S_{s}(x, \chi)}{\log x}+O \left(\frac{1}{(\log x)^{2}} \right) \Bigg)\\ \times
	\begin{cases}
	\sqrt{2} & \text{if $s = 1/2$ and $\chi^{2} = 1$},\\
	\exp \left(\delta_{\chi^{2} = 1} \dfrac{x^{1-2s}(2x^{s-1/2}-1)}{2(2s-1) \log x} \right) & \text{otherwise},
	\end{cases}
\end{multline*}
\item[Case I\hspace{-.1mm}I\hspace{-.1mm}I. $\Re(s) > 1/2$] 
\begin{multline*}
\prod_{p \leqslant x} (1-\chi(p) p^{-s})^{-1}
 = L(s, \chi) \exp \Bigg(\frac{1}{\varphi(q)} \ \sideset{}{^{\ast}} \sum_{a \tpmod q} \chi(a) \Li((\varphi(q) \vartheta(x; q, a))^{1-s})\\
 + \frac{(2s-1+\delta_{\chi^{2} = 1})x^{1/2-s}}{(2s-1) \log x}+\frac{S_{s}(x, \chi)}{\log x}
 + O \left(\frac{x^{1/2-\sigma}}{(\log x)^{2}} \right) \Bigg).
\end{multline*}
\end{description}
\end{theorem}

Disregarding the exponential multipliers in Theorem~\ref{aim}, it is conjectured for $\chi \ne 1$ that
\begin{equation}\label{RH}
\lim_{x \to \infty} \prod_{p \leqslant x} (1-\chi(p)p^{-s})^{-1} = L(s, \chi)
\end{equation}
on the half-plane $\Re(s) > 1/2$. Conrad~\cite{Conrad2005} has shown that GRH is equivalent to~\eqref{RH}. 
As a strengthened version of~\eqref{RH}, DRH asserts the following in the case of Dirichlet $L$-functions.
\begin{conjecture}[DRH for Dirichlet $L$-functions]\label{DRH}
If $\chi \ne 1$ and a complex number $s$ is on the critical line~$\Re(s) = 1/2$ with $m$ the order of vanishing of $L(s, \chi)$, we have
\begin{equation}\label{DirichletL}
\lim_{x \to \infty} \left((\log x)^{m} \prod_{p \leqslant x} (1-\chi(p)p^{-s})^{-1} \right)
 = \frac{L^{(m)}(s,\chi)}{e^{m\gamma}m!} \times
\begin{cases}
	\sqrt{2} & \text{if $s = 1/2$ and $\chi^{2} = 1$},\\
	1 & \text{otherwise}.
\end{cases}
\end{equation}
\end{conjecture}
The two statements that the limit on the left-hand side of~\eqref{DirichletL} exists for \textit{some} $s$ on $\Re(s) = 1/2$ and that it exists for \textit{every} $s$ on $\Re(s) = 1/2$ are equivalent. Moreover, the conjecture~\eqref{DirichletL} is known to be equivalent to
\begin{equation}\label{Conrad}
\vartheta(x; q, a)-\frac{x}{\varphi(q)} = o(\sqrt{x} \log x).
\end{equation}
This bound is better than what one can reach under GRH. Case I\hspace{-.1mm}I of Theorem~\ref{aim} shows that DRH holds when~the error term in the Prime Number Theorem in arithmetic progressions is bounded as in~\eqref{Conrad}.

Conrad~\cite{Conrad2005} considered partial Euler products for various $L$-functions along their critical line and demystified the $\sqrt{2}$ phenomenon. He found the equivalence between the Euler product asymptotics and~the estimate~$\psi_{L}(x) \coloneqq \sum_{N \frak{p}^{k} \leqslant x} (\alpha_{\frak{p}, 1}^{k}+\dots+\alpha_{\frak{p}, d}^{k}) \log \mathrm{N} \frak{p} = o(\sqrt{x} \log x)$, which is stronger than GRH. Given an elliptic curve $E/\Q$ with $N$ the conductor, Kuo--Murty~\cite{KuoMurty2005} established the equivalence between the Birch and Swinnerton-Dyer conjecture and the bound $\sum_{n \leqslant x} \tilde{c}_{n} = o(x)$. Here $\tilde{c}_{n}$ signifies that
\begin{equation*}
\tilde{c}_{n} = 
	\begin{cases}
	\frac{\alpha_{p}^{k}+\beta_{p}^{k}}{k} & \text{$n = p^{k}$ for $p \nmid N$},\\
	0 & \text{otherwise},
	\end{cases}
\end{equation*}
with $\alpha_{p}$ and $\beta_{p}$ the Frobenius eigenvalues at $p$. Akatsuka~\cite{Akatsuka2017} has studied DRH for the Riemann zeta function. With a simple pole of $\zeta(s)$ in mind, DRH is equivalent to the estimate $\vartheta(x)-x = o(\sqrt{x} \log x)$.


\subsection*{Acknowledgement}
This article is an outgrowth of the author's collaborative work with Koyama~\cite{KanekoKoyama2018}, where~we have studied  Euler products of Selberg zeta functions in the critical strip. The author would like to thank Shin-ya Koyama and Nobushige Kurokawa for illuminative discussions.

\section{Preliminaries}\label{preliminaries}

\subsection{The work of Ramanujan}\label{the-work-of-Ramanujan}
Ramanujan~\cite{Ramanujan1997} extended, beyond the boundary, the result of~Mertens to~$s > 0$ in the process of obtaining the maximal order of the divisor function. His formula then asserts that if the~Riemann Hypothesis (RH) for $\zeta(s)$ is assumed, we have
\begin{equation}\label{Ramanujan}
\prod_{p \leqslant x} (1-p^{-s})^{-1}
 = -\zeta(s) \exp \left(\Li(\vartheta(x)^{1-s})+\frac{2s x^{\frac{1}{2}-s}}{(2s-1) \log x}
 + \frac{S_{s}(x)}{\log x}+O \left(\frac{x^{\frac{1}{2}-s}}{(\log x)^{2}} \right) \right)
\end{equation}
in $1/2 < s < 1$, where $\vartheta(x) = \sum_{p \leqslant x} \log p$ is the Chebyshev function and
\begin{equation}\label{zeroes}
S_{s}(x) = -s \sum_{\zeta(\rho) = 0} \frac{x^{\rho-s}}{\rho(\rho-s)}.
\end{equation}
His method is contained in the article~\cite{Ramanujan1915} entitled `Highly Composite Numbers' and the rest~\cite{Ramanujan1997} was~published in 1997. There are two manuscripts by him (handwritten by Watson) on sums involving primes. These are found on pages 228--232 in \cite{Ramanujan1988}. His original manuscripts are stored in the library at Trinity College, Cambridge.

\subsection{Prime Number Theorem in arithmetic progressions}\label{prime-number-theorem-in-arithmetic-progressions}
Let $\chi$ be a primitive Dirichlet character modulo $q$. Then a Dirichlet $L$-function~$L(s, \chi)$ satisfies the functional equation
\begin{equation}\label{functional}
\Lambda(s, \chi) \coloneqq \left(\frac{q}{\pi} \right)^{s/2} \Gamma \left(\frac{s+\nu}{2} \right) L(s, \chi)
 = \epsilon(\chi) \Lambda(1-s, \overline{\chi}),
\end{equation}
where $\nu = (1-\chi(-1))/2$ and $\epsilon(\chi) = i^{-\nu} q^{-1/2} \tau(\chi)$ with the Gauss sum $\tau(\chi)$. The completed $L$-function~$\Lambda(s, \chi)$ has a meromorphic continuation to $\C$ and is entire if $\chi \ne 1$. Let $\pi(x; q, a)$ be the number of primes $p$ up to $x$ belonging to the arithmetic progression $a \tpmod q$. In particular, we set $\pi(x) \coloneqq \pi(x; 1, 1)$ and this abbreviation also applies to other functions below. The Prime Number Theorem in arithmetic progressions shows that in any residue class $a \tpmod q$, the primes are equidistributed amongst the plausible arithmetic progressions modulo~$q$:
\begin{equation}
\pi(x; q, a) \sim \frac{\pi(x)}{\varphi(q)}
\end{equation}
as $x \to \infty$ whenever $(a, q) = 1$ and $q \geqslant 1$. An important question here is the uniformity in $q$, which is relevant to the distribution of zeroes of $L(s, \chi)$. Moreover, the Siegel--Walfisz theorem~\cite{Siegel1935,Walfisz1936} asserts that
\begin{equation}\label{Siegel-Walfisz}
\pi(x; q, a) = \frac{1}{\varphi(q)} \Li(x)+O(x \exp(-c \sqrt{\log x}))
\end{equation}
for any $q \leqslant (\log x)^{A}$ where $c$ and the implicit constant depend on $A$ alone (not effectively computable if~$A \geqslant 2$). Nonetheless, it is beneficial to weaken the restriction on $q$ for applications. The assumption of GRH yields~\eqref{Siegel-Walfisz} in a much wider regime $q \leqslant \sqrt{x}/(\log x)^{2+\epsilon}$. Since we assume GRH throughout this article, such a restriction on $q$ adheres to our discussion. It is conjectured that the following asymptotic is available:
\begin{equation*}
\pi(x; q, a) = \frac{1}{\varphi(q)} \Li(x)+O(x^{1/2+\epsilon})
\end{equation*}
uniformly for $q \leqslant x^{1/2-\epsilon}$. For notational convenience, we introduce
\begin{equation*}
E(x; q, a) = \frac{x}{\varphi(q)}-\vartheta(x; q, a).
\end{equation*}
It is well known that the bound $|E(x; q, a)| \ll \sqrt{x}(\log x)^{2}$ is tantamount to GRH.

\subsection{Summation formul{\ae}}\label{summation-formulae}
Following Ramanujan, we start with considering the partial summation that if $\Phi^{\prime}(x)$ is a continuous function, then
\begin{equation*}
\Phi(p_{1}) \log p_{1}+\Phi(p_{2}) \log p_{2}+\cdots+\Phi(p_{n}) \log p_{n}
 = \Phi(x) \vartheta(x; q, a)-\int_{p_{1}}^{x} \Phi^{\prime}(t) \vartheta(t; q, a) dt,
\end{equation*}
where $p_{1} < p_{2} < \cdots < p_{n}$ is an ascending sequence of consecutive primes of the form $mq+a$ and $p_{1}$ (resp.~$p_{n}$) stands for the smallest (resp. largest) prime below $x$ of such a form. Integrating by parts gives
\begin{equation*}
\Phi(x) \vartheta(x; q, a)-\int_{p_{1}}^{x} \Phi^{\prime}(t) \vartheta(t; q, a) dt
 = \mathrm{const}+\frac{1}{\varphi(q)} \int_{p_{1}}^{x} \Phi(t) dt-\Phi(x) E(x; q, a)+\int_{p_{1}}^{x} \Phi^{\prime}(t) E(t; q, a) dt
\end{equation*}
where `const' depends on $a$, $q$ and $\Phi$. In what follows, we assume GRH for Dirichlet $L$-functions, which allows us to work in the regime $q \leqslant \sqrt{x}/(\log x)^{2+\epsilon}$. Taylor's theorem then yields that
\begin{equation*}
\int_{p_{1}}^{\varphi(q) \vartheta(x; q, a)} \Phi(t) dt
 = \int_{p_{1}}^{x} \Phi(t) dt-\Phi(x) \varphi(q) E(x; q, a)+\frac{1}{2}(\varphi(q) E(x; q, a))^{2} \Phi^{\prime}(x+O(\sqrt{x}(\log x)^{2})).
\end{equation*}
Gathering together these formul{\ae}, one sees that
\begin{multline}\label{summation-formula}
\Phi(p_{1}) \log p_{1}+\Phi(p_{2}) \log p_{2}+\dots+\Phi(p_{n}) \log p_{n}\\
 = \mathrm{const}+\frac{1}{\varphi(q)} \int_{p_{1}}^{\varphi(q) \vartheta(x; q, a)} \Phi(t) dt
 + \int_{p_{1}}^{x} \Phi^{\prime}(t) E(t; q, a) dt
 - \frac{1}{2} \varphi(q) E(x; q, a)^{2} \Phi^{\prime}(x+O(\sqrt{x}(\log x)^{2})),
\end{multline}

\section{Proof of Theorem~\ref{aim}}\label{proof}
In this section, we establish Theorem~\ref{aim}.

\subsection{Partial Euler products for Dirichlet $L$-functions}\label{PartialEulerProd}
Let $\chi$ be a primitive character modulo $q$. We exploit~the summation formula~\eqref{summation-formula} $\varphi(q)$ times. For our purpose, we assume that $\Phi(x) = \chi(a)/(x^{s}-\chi(a))$ for each $a$ with $(a, q) = 1$. We also assume for the sake of simplicity that $\Re(s) > 0$ throughout this article. Hence one derives
\begin{multline}\label{log}
\frac{\chi(2) \log 2}{2^{s}-\chi(2)}+\frac{\chi(3) \log 3}{3^{s}-\chi(3)}
 + \frac{\chi(5) \log 5}{5^{s}-\chi(5)}+\dots+\frac{\chi(p) \log p}{p^{s}-\chi(p)}\\
 = \mathrm{const}+\frac{1}{\varphi(q)} \ \sideset{}{^{\ast}} \sum_{a \tpmod q} 
\int_{p_{1}}^{\varphi(q) \vartheta(x; q, a)} \frac{\chi(a) dt}{t^{s}-\chi(a)}\\
 - s \ \sideset{}{^{\ast}} \sum_{a \tpmod q} \chi(a) 
\int_{p_{1}}^{x} \frac{E(t; q, a) dt}{t^{1-s}(t^{s}-\chi(a))^{2}}+O(x^{-s}(\log x)^{4}),
\end{multline}
where $p$ signifies the largest prime below $x$. Moreover, the explicit formula (cf.~\cite{IwaniecKowalski2004}) renders that
\begin{equation}\label{Explicit}
\varphi(q) E(t; q, a) = \delta_{2}(q, a) \sqrt{t}+\delta_{3}(q, a) \sqrt[3]{t}
 + \sum_{\psi \tpmod q} \overline{\psi}(a) \sum_{L(\rho, \psi) = 0} \frac{t^{\rho}}{\rho}
 - \sum_{\psi \tpmod q} \overline{\psi}(a) \sum_{L(\rho, \psi) = 0} \frac{t^{\rho/2}}{\rho}+O(t^{1/5}),
\end{equation}
where
\begin{equation*}
\delta_{m}(q, a) = \# \{x \tpmod q: x^{m} \equiv a \tpmod q \}, 
\end{equation*}
the outer sums over $\psi$ range over all Dirichlet characters modulo $q$ and the inner sums are over nontrivial zeroes of $L(s, \psi)$. Note that the number of primitive Dirichlet characters modulo $q$ is given by $\sum_{dr = q} \mu(d) \phi(r)$. Using the Chinese Remainder Theorem, the function $\delta_{2}(q, a)$ is multiplicative if $q = 2^{n_{2}} 3^{n_{3}} 5^{n_{5}} \cdots p_{1}^{n_{p_{1}}}$. To be~accurate, counting the solutions to $x^{2} \equiv a \pmod q$, we infer for $a$ with $(a/q) = 1$ that
\begin{equation*}
\delta_{2}(q, a) = \prod_{p} \delta(p^{n_{p}}, a)
 = 2^{\omega(q)-1} \times
	\begin{cases}
	4 & \quad \text{if $n_{2} \geqslant 3$},\\
	2 & \quad \text{if $n_{2} = 2$},\\
	1 & \quad \text{if $n_{2} = 1$},\\
	2 & \quad \text{otherwise},
	\end{cases}
\end{equation*}
where $\omega(q)$ is the number of different prime factors of $q$. Then it turns out that the contribution from the fourth term on the right hand side of~\eqref{Explicit} is bounded as
\begin{equation*}
\sum_{\psi \tpmod q} \overline{\psi}(a) \sum_{L(\rho, \psi) = 0} \frac{1}{\rho}
\int_{p_{1}}^{x} \frac{x^{s+\rho/2-1} dx}{(x^{s}-\chi(a))^{2}}
 \ll \Bigg|\sum_{\psi \tpmod q} \sum_{L(\rho, \psi) = 0} \frac{x^{\rho/2-s}}{\rho(\rho/2-s)} \Bigg| \ll x^{1/4-\sigma},
\end{equation*}
Since $(x^{s}-\chi(a))^{-2} = x^{-2s}+O(x^{-3\sigma})$, the contribution of the third term on the right hand side of~\eqref{Explicit} becomes
\begin{multline*}
\ \sideset{}{^{\ast}} \sum_{a \tpmod q} \chi(a) 
\sum_{\psi \tpmod q} \overline{\psi}(a) \sum_{L(\rho, \psi) = 0} \frac{1}{\rho}
\int_{p_{1}}^{x} \frac{x^{s+\rho-1} dx}{(x^{s}-\chi(a))^{2}}\\
 = \mathrm{const}+\ \sideset{}{^{\ast}} \sum_{a \tpmod q} \chi(a) \sum_{\psi \tpmod q} \overline{\psi}(a) 
\sum_{L(\rho, \psi) = 0} \frac{x^{\rho-s}}{\rho(\rho-s)}+O(x^{1/2-2s}).
\end{multline*}
Hence we recall the definition
\begin{equation*}
S_{s}(x, \chi) = -\frac{s}{\varphi(q)} \ \sideset{}{^{\ast}} \sum_{a \tpmod q} 
\chi(a) \sum_{\psi \tpmod q} \overline{\psi}(a) \sum_{L(\rho, \psi) = 0} \dfrac{x^{\rho-s}}{\rho(\rho-s)}
\end{equation*}
to arrive at the expression
\begin{multline}\label{logarithm}
\frac{\chi(2) \log 2}{2^{s}-\chi(2)}+\frac{\chi(3) \log 3}{3^{s}-\chi(3)}+\dots
 + \frac{\chi(p) \log p}{p^{s}-\chi(p)}\\
 = \mathrm{const}+\frac{1}{\varphi(q)} \ \sideset{}{^{\ast}} \sum_{a \tpmod q} 
\int_{p_{1}}^{\varphi(q) \vartheta(x; q, a)} \frac{\chi(a) dt}{t^{s}-\chi(a)}
 - \sideset{}{^{\ast}} \sum_{a \tpmod q} s \int_{p_{1}}^{x} \chi(a) \frac{\delta_{2}(q, a) \sqrt{t}\
 + \delta_{3}(q, a) \sqrt[3]{t}}{t^{1-s}(t^{s}-\chi(a))^{2}} dt\\
 + S_{s}(x, \chi)+O(x^{1/2-2s}+x^{1/4-s}).
\end{multline}

\subsection{Completion of the proof}\label{completion-of-the-proof}
One can establish the following lemma:
\begin{lemma}\label{orthogonality}
Let $\chi$ be a Dirichlet character modulo $q$ and let $m \geqslant 1$ be a positive integer. We then have
\begin{equation*}
\ \sideset{}{^{\ast}} \sum_{a \tpmod q} \chi(a) \delta_{m}(q, a) = 
	\begin{cases}
	\varphi(q) & \text{if $\chi^{m} = 1$},\\
	0 & \text{otherwise},
	\end{cases}
\end{equation*}
\end{lemma}

\begin{proof}
The proof relies on the identity
\begin{equation*}
\ \sideset{}{^{\ast}} \sum_{a \tpmod q} \chi(a) \delta_{m}(q, a) = \ \sideset{}{^{\ast}} \sum_{a \tpmod q} \chi(a)^{m}
\end{equation*}
to which we can apply the classical orthogonality relation.
\end{proof}

We use Lemma~\ref{orthogonality} after expanding the integrands of the two integrals on the right hand side of~\eqref{logarithm}~respectively. Upon truncating unnecessary terms, the sum of the contributions from these integrals equals
\begin{multline*}
\mathrm{const}+\frac{1}{\varphi(q)} \ \sideset{}{^{\ast}} \sum_{a \tpmod q} \left(\frac{\chi(a)}{1-s} (\varphi(q) \vartheta(x; q, a))^{1-s}
 + \frac{\chi(a)^{2}}{1-2s} (\varphi(q) \vartheta(x; q, a))^{1-2s} \right)\\
 + \frac{x^{1-\varphi(q) s}}{1-\varphi(q) s}+\dots+\frac{x^{1-ns}}{1-ns}
 - \delta_{\chi^{2} = 1} \dfrac{2s x^{1/2-s}}{1-2s}-\delta_{\chi^{3} = 1} \dfrac{3s x^{1/3-s}}{1-3s}
 - \delta_{\chi^{2} = 1} \dfrac{4s x^{1/2-2s}}{1-4s}+O(x^{1/2-2s}),
\end{multline*}
where $n$ is the largest multiple of $\varphi(q)$ not exceeding $\left[2+1/2s \right]$. It therefore follows that
\begin{multline}\label{asymptotic}
\frac{\chi(2) \log 2}{2^{s}-\chi(2)}+\frac{\chi(3) \log 3}{3^{s}-\chi(3)}
 + \cdots+\frac{\chi(p) \log p}{p^{s}-\chi(p)}\\
 = -\frac{L^{\prime}(s, \chi)}{L(s, \chi)}
 + \frac{1}{\varphi(q)} \ \sideset{}{^{\ast}} \sum_{a \tpmod q} \left(\frac{\chi(a)}{1-s} (\varphi(q) \vartheta(x; q, a))^{1-s}
 + \frac{\chi(a)^{2}}{1-2s} (\varphi(q) \vartheta(x; q, a))^{1-2s} \right)
 + \frac{x^{1-\varphi(q) s}}{1-\varphi(q) s}+\dotsb\\
 + \frac{x^{1-ns}}{1-ns}-\delta_{\chi^{2} = 1} \dfrac{2s x^{1/2-s}}{1-2s}-\delta_{\chi^{3} = 1} \dfrac{3s x^{1/3-s}}{1-3s}
 - \delta_{\chi^{2} = 1} \dfrac{4s x^{1/2-2s}}{1-4s}+S_{s}(x, \chi)+O(x^{1/2-2\sigma}+x^{1/4-\sigma}).
\end{multline}
We note that $\mathrm{const} = -L^{\prime}(s, \chi)/L(s, \chi)$ is justified by the work of Conrad~\cite{Conrad2005}.
If we consider the~Euler~products at $s = 1, 1/2, 1/3$ and $1/4$, we must take the limit of the right hand side of~\eqref{asymptotic}. In order to finish our proof,~one should replace $s = \sigma+it$ with $u+it$ and integrate the asymptotic formula~\eqref{asymptotic} in $u$ once from $\infty$ to $\sigma$, obtaining
\begin{multline}\label{log}
\log \prod_{p \leqslant x} (1-\chi(p) p^{-s})
 = -\log L(s, \chi)+\frac{1}{\varphi(q)} \ \sideset{}{^{\ast}} \sum_{a \tpmod q} \chi(a) \Li((\varphi(q) \vartheta(x; q, a))^{1-s})\\
 - \delta_{\chi^{2} = 1} \dfrac{1}{2} \Li(x^{1-2s})+\delta_{\chi^{2} = 1} \dfrac{1}{2} \Li(x^{1/2-s})
 - \dfrac{1}{\varphi(q)} \Li(x^{1-\varphi(q) s})-\cdots\\
 - \dfrac{1}{n} \Li(x^{1-ns})-\frac{x^{1/2-s}+S_{s}(x, \chi)}{\log x}
 + O \left(\frac{x^{1/2-\sigma}}{(\log x)^{2}} \right),
\end{multline}
where we have used that
\begin{equation*}
\int_{\infty}^{\sigma} \frac{x^{a+b(u+it)}}{a+b(u+it)} du = \frac{1}{b} \Li(x^{a+bs}), \qquad
\int_{\infty}^{\sigma} S_{u+it}(x, \chi) du = -\frac{S_{s}(x, \chi)}{\log x}+O \left(\frac{x^{1/2-\sigma}}{(\log x)^{2}} \right).
\end{equation*}
Exponentiating~\eqref{log} and classifying our resulting formula into the three cases $0 < \Re(s) < 1/2$, $\Re(s) = 1/2$ and $\Re(s) > 1/2$, we can deduce the desired formula. This concludes the proof of Theorem~\ref{aim}. \qed


\end{document}